\documentclass{amsart}

\usepackage[english]{babel}

\usepackage[letterpaper,top=2cm,bottom=2cm,left=3cm,right=3cm,marginparwidth=1.75cm]{geometry}

\usepackage{amsmath}
\usepackage{amssymb}
\usepackage{graphicx}
\usepackage[colorlinks=true, allcolors=blue]{hyperref}

\usepackage{setspace}
\usepackage{blindtext}
\newtheorem{corollary}{Corollary}[section]
\newtheorem{definition}{Definition}[section]
\newtheorem{lemma}{Lemma}[section]

\title{A Non-Abelian Approach to Riemann Surfaces\\ Part I: 
Wronskian Geometry}
\author{Mehrzad Ajoodanian}

\begin{document}
\maketitle
\begin{abstract}
We study projectively flat holomorphic vector bundles over Riemann surfaces. To each such bundle, we naturally assign a \emph{Wronskian line bundle}. The main idea is a notion of the division of two meromorphic sections. Abel's identity is interpreted as the first Chern class of the Wronskian line bundle.
\end{abstract}

\section{Introduction}

Let $X$ be a Riemann surface and let $\mathcal{L}$ be a holomorphic line bundle on $X$.  
If $A$ and $B \neq 0$ are two meromorphic sections of $\mathcal{L}$, their quotient
\[
\frac{A}{B}
\]
is a meromorphic function on $X$.  
This simple observation underlies some fundamental constructions in complex geometry: it allows one to describe explicitly the divisor class of $\mathcal{L}$ and to compute its first Chern class.  
Given a meromorphic section $A\neq 0$ of $\mathcal{L}$, 
we cover $X$ with the union of appropriate open sets $U_i$ such that $A$ has no zeros and no poles on the intersection $U_i\cap U_j$ for $i\neq j$. Let $A_i=A|_{U_i}$. Then the invertible holomorphic ratio $\frac{A_i}{A_j}$ defines a cocycle for $\mathcal{L}$ and the logarithmic differential $d\log \frac{A_i}{A_j}$ provides a 1-\v{C}ech cocycle with coefficients in the sheaf of differentials. This class is independent of the choice of the section $A$ and represents the first Chern class of $\mathcal{L}$.

\medskip
This is a well-known beautiful Abelian story. Our goal is to turn the story non-Abelian. To begin with, we pose the following questions: can we define an analogous quotient for meromorphic sections of a vector bundle?  
If such a quotient exists, what kind of geometric object does it produce?  
In this article, we propose coherent answers to these natural questions for \emph{projectively flat} holomorphic vector bundles on Riemann surfaces that we regard as a non-Abelian version of a line bundle.

\medskip

More precisely, let $\mathcal{V}$ be a projectively flat holomorphic vector bundle of rank $n$ over $X$.  
We define the quotient of a meromorphic section $A$ of $\mathcal{V}$ by a generic meromorphic section $B$ (to be defined precisely later) as a meromorphic section of the endomorphism bundle $\mathrm{End}(\mathcal{V})$.  
This trick helps extend the familiar notion of the ratio of two sections of a line bundle to the setting of higher rank bundles, where the quotient becomes an endomorphism (non-Abelian) rather than a scalar (Abelian).

\medskip

This non-Abelian quest leads us to a new geometric invariant:  
the \emph{Wronskian line bundle} $w(\mathcal{V})$ associated to a projectively flat vector bundle $\mathcal{V}$.  
The construction of $w(\mathcal{V})$ is based on the classical Wronskian determinant, familiar in the theory of ordinary differential equations. 
We then build an algebraic-geometric setting for the Wronskian matrix, not just the determinant as it is often considered, independent of coordinates and local charts. 
\medskip

Historically, the Wronskian determinant was introduced by Józef Hoene-Wroński in 1812 as a criterion for testing the linear independence of analytic functions.  
Later, its role became central in the works of Abel, Liouville, and Jacobi, who established deep relationships between the Wronskian, differential equations, and the geometry of function spaces.  
In this article, we reinterpret some of these classical results geometrically in terms of the Wronskian line bundle and its first Chern class.

\medskip

We then explore whether the Wronskian line bundle arises as the determinant of a natural rank-$n$ vector bundle $W(\mathcal{V})$, canonically associated to $\mathcal{V}$.  
We show that such a bundle can indeed be constructed, though not uniquely, and we investigate its dependence on the choice of a generic section.  

\medskip

Together with Amir Jafari, in a paper on Schwarzian geometry, we apply the framework discussed here to arrive at a new generalization of the Schwarzian derivative on Riemann surfaces. 

\medskip

We are now ready to embark on a non-Abelian journey to study Riemann surfaces carrying a projectively flat vector bundle. We shall rendezvous with Abel, Wro\'nski, Liouville and Chern, among others, along the way.

\section{The Wronskian}

The \emph{Wronskian} is a classical tool in the theory of differential equations. It provides a determinant-based criterion for determining whether a collection of analytic functions is linearly independent. In this section, we recall its definition and describe several of its key properties, presented in a fashion that suits our purpose.

\subsection{Definition and Basic Properties}

Let $U \subseteq \mathbb{C}$ be an open subset, and let $V$ be a complex vector space of dimension $n$. Given a meromorphic function $A: U \to V$ and a choice of basis for $V$, we can write
\[
A = (a_1, a_2, \dots, a_n),
\]
where each $a_i$ is a meromorphic function on $U$.

The \emph{Wronskian matrix} associated to $A$ (relative to the chosen basis) is the $n \times n$ matrix of meromorphic functions
\[
W(A) =
\begin{bmatrix}
a_1 & a_2 & \dots & a_n \\
a_1' & a_2' & \dots & a_n' \\
\vdots & \vdots & \ddots & \vdots \\
a_1^{(n-1)} & a_2^{(n-1)} & \dots & a_n^{(n-1)}
\end{bmatrix},
\]
where $a_i^{(j)}$ denotes the $j$-th derivative of $a_i$. The \emph{Wronskian determinant} of $A$ is then defined by
\[
w(A) = \det W(A).
\]

A change of basis in $V$ acts naturally on the Wronskian: if $T \in GL_n(\mathbb{C})$ represents a change of basis, then
\[
W(A) \mapsto W(A)T.
\]

\subsection{Change of Coordinates}

The Wronskian also transforms predictably under a holomorphic change of coordinates. If $\lambda: U_1 \to U$ is a holomorphic map, then by the chain rule,
\[
W(A \circ \lambda) = \Lambda_n(\lambda) \cdot (W(A) \circ \lambda),
\]
where $\Lambda_n(\lambda)$ is a lower triangular matrix whose diagonal entries are
\[
1, \lambda', (\lambda')^2, \dots, (\lambda')^{n-1},
\]
and whose subdiagonal entries involve higher derivatives of $\lambda$.

For example, for $n=3$ and $n=4$, we obtain
\[
\Lambda_3(\lambda)=
\begin{bmatrix}
1 & 0 & 0\\
0 & \lambda' & 0\\
0 & \lambda'' & (\lambda')^2
\end{bmatrix},
\qquad
\Lambda_4(\lambda)=
\begin{bmatrix}
1 & 0 & 0 & 0\\
0 & \lambda' & 0 & 0\\
0 & \lambda'' & (\lambda')^2 & 0\\
0 & \lambda''' & 3\lambda'\lambda'' & (\lambda')^3
\end{bmatrix}.
\]

In general, $\Lambda_n(\lambda)$ can be expressed explicitly using the classical Faà di Bruno formula (1855):
\[
(A \circ \lambda)^{(n)}
= \sum
\frac{n!}{m_1! \cdots m_n!}
A^{(m_1+\cdots+m_n)}
\prod_{j=1}^{n}
\left( \frac{\lambda^{(j)}}{j!} \right)^{m_j},
\]
where the sum is taken over all nonnegative tuples $(m_1, \dots, m_n)$ satisfying
$\sum_{j=1}^n j m_j = n$. This provides a constructive description of $\Lambda_n(\lambda)$, though we will not need the explicit form here.

\begin{definition}
A meromorphic map $B: U \to V$ is called \emph{generic} if $w(B) \not\equiv 0$, i.e., its Wronskian determinant is not identically zero.
\end{definition}

Equivalently, a meromorphic map $B$ is generic if and only if its components are linearly independent. This notion is independent of the chosen basis for $V$ and of the coordinate on $U$.

\subsection{The Wronskian Quotient}

\begin{definition}
Given two meromorphic maps $A, B: U \to V$, with $B$ generic, we define their \emph{Wronskian quotient} as the meromorphic map
\[
\frac{A}{B}: U \longrightarrow \mathrm{End}(V), \qquad
\frac{A}{B} = W(B)^{-1} W(A).
\]
\end{definition}

\begin{lemma}
The map $\frac{A}{B}$ is independent of the choice of basis for $V$ and of the coordinate on $U$.
\end{lemma}

\begin{proof}
If $T$ is the change of basis matrix and $\Lambda_n$ the corresponding coordinate change matrix, then
\[
\frac{A}{B} \mapsto T^{-1} W(B)^{-1} \Lambda_n^{-1} \Lambda_n W(A) T = T^{-1}\left(\frac{A}{B}\right)T,
\]
showing that $\frac{A}{B}$ transforms by conjugation and hence defines a well-defined endomorphism of $V$.
\end{proof}

\subsection{Algebraic Properties of the Quotient}

The quotient $\frac{A}{B}$ satisfies several natural algebraic identities.

\begin{lemma}
Let $A, B, C: U \to V$ be generic meromorphic maps. Then:
\[
\frac{A}{B} \cdot \frac{C}{A} \cdot \frac{B}{C} = 1.
\]
Furthermore, for any meromorphic function $f$ that is not identically zero,
\[
\frac{fA}{fB} = \frac{A}{B}.
\]
\end{lemma}

\begin{proof}
The first identity follows directly from the definition.  
For the second, observe that by the product rule,
\[
W(fA) = \Phi_n(f) \cdot W(A),
\]
where $\Phi_n(f)$ is a lower triangular matrix whose diagonal entries are $f$ and whose subdiagonal entries contain derivatives of $f$. For example,
\[
\Phi_3(f) =
\begin{bmatrix}
f & 0 & 0\\
f' & f & 0\\
f'' & 2f' & f
\end{bmatrix},
\qquad
\Phi_4(f) =
\begin{bmatrix}
f & 0 & 0 & 0\\
f' & f & 0 & 0\\
f'' & 2f' & f & 0\\
f''' & 3f'' & 3f' & f
\end{bmatrix}.
\]
In general, $(\Phi_n(f))_{ij} = 0$ for $i < j$ and
\[
(\Phi_n(f))_{ij} = \binom{i-1}{i-j} f^{(i-j)} \quad \text{for } i \ge j.
\]
Hence,
\[
W(fB)^{-1} W(fA)
= W(B)^{-1} \Phi_n(f)^{-1} \Phi_n(f) W(A)
= W(B)^{-1} W(A).
\]
\end{proof}

\subsection{The Rank of the Derivative of the Quotient as an Invariant}

The rank of the derivative is well defined and is independent of both the change of chart and frame.

\begin{lemma}
Let $A, B: U \to V$ be meromorphic maps, with $B$ generic.  
Then the derivative of
\[
\frac{A}{B}: U \longrightarrow \mathrm{End}(V)
\]
has rank at most one at every point.
\end{lemma}

\begin{proof}
Let $Q = W(B)^{-1} W(A)$, so that $W(A) = W(B) Q$.  
The $i$-th row of $W(A)$ is $A^{(i-1)}$, and similarly the $i$-th row of $W(B)$ is $B^{(i-1)}$.  
Thus, for each $i$,
\[
A^{(i-1)} = B^{(i-1)} Q.
\]
Differentiating and using $A^{(i)} = B^{(i)} Q$, we obtain
\[
B^{(i-1)} Q' = 0 \quad \text{for } i = 1, \dots, n-1.
\]
Since $B$ is generic, the vectors $B^{(0)}, \dots, B^{(n-2)}$ are linearly independent.  
Hence the nullity of $Q'$ is at least $(n-1)$, so $\mathrm{rank}(Q') \le 1$.
\end{proof}

\subsection{Ordinary Differential Equations}

The Wronskian appears naturally in the theory of linear homogeneous differential equations.  
Let $A: U \to V$ be a generic meromorphic map. Then the functions
\[
A, A', \dots, A^{(n-1)}
\]
are linearly independent, and thus there exist meromorphic functions $p_1, \dots, p_{n-1}$ such that $A$ satisfies the linear ODE of order $n$:
\[
A^{(n)} = p_1 A^{(n-1)} + \dots + p_{n-1} A.
\]

\begin{lemma}[Abel's identity]
The Wronskian $w(A) = \det W(A)$ satisfies the first-order differential equation
\[
w' = p_1 w.
\]
\end{lemma}

\begin{proof}
Abel's identity easily follows from Cramer's rule in linear algebra, together with the differentiation rule of determinants.
\end{proof}
There is an underlying geometric interpretation of Abel's formula that adds flavor to our discussion.
Abel shows that $p_1 = d(\log w)$, i.e., $p_1$ defines a meromorphic $1$-form on $U$.  
Later, we will interpret this $1$-form as the \emph{first Chern class} of the Wronskian line bundle.

\subsection{Liouville’s Formula, Abel’s Identity, and the First Chern Class}

A classical result of Liouville provides a non-Abelian generalization of Abel’s identity.  
If $\Phi: U \to GL_n(\mathbb{C})$ is a meromorphic map, then
\[
\mathrm{Tr}(\Phi^{-1} d\Phi) = d\log(\det \Phi).
\]
That is, the trace of the Maurer–Cartan form $\Phi^{-1} d\Phi$ equals the differential of the logarithm of the determinant of $\Phi$.  
Abel’s identity follows 
from Liouville's formula, and the following beautiful picture emerges. Liouville provides a \v{C}ech representative for the first Chern class of a holomorphic vector bundle in terms of the trace of Maurer–Cartan forms of the transition matrices. That in turn defines the first Chern class of the vector bundle, which equals the first Chern class of the associated determinant line bundle. 

\medskip
Next we upgrade from local charts and vector spaces to Riemann surfaces and projectively flat vector bundles.

\section{Projectively Flat Vector Bundles on Riemann Surfaces}

Let $X$ be a Riemann surface and let $\mathcal{V}$ be a holomorphic vector bundle over $X$.  
Our goal in this section is to extend the local computations carried out earlier to a \emph{global} setting by considering a generic section $A$ of $\mathcal{V}$.  
Before doing so, we recall the precise definition of a projectively flat bundle.

\begin{definition}
A holomorphic vector bundle $\mathcal{V}$ of rank $n$ over a Riemann surface $X$ is called \emph{projectively flat} (abbreviated \emph{pf}) if there exists an open cover $\{U_i\}$ of $X$ and trivializations of $\mathcal{V}$ over each $U_i$ such that the transition functions on overlaps $U_i \cap U_j$ are of the form
\[
\Phi_{ij} = \lambda_{ij} \cdot T_{ij},
\]
where $\lambda_{ij}: U_i \cap U_j \to \mathbb{C}^\times$ is a holomorphic function and $T_{ij} \in GL_n(\mathbb{C})$ is a constant matrix.
\end{definition}

Intuitively, this means that the bundle $\mathcal{V}$ becomes locally trivial once we ignore an overall scalar factor.  
The transition functions differ from being globally constant only by the multiplicative holomorphic scalars $\lambda_{ij}$, which capture the “projective” twisting of the bundle.  
Thus, a projectively flat bundle can be viewed as a vector bundle whose projectivization $\mathbb{P}(\mathcal{V})$ admits a \emph{flat connection}.

Every line bundle (rank one bundle) is automatically projectively flat, since its transition functions already take values in $\mathbb{C}^\times$.  
The notion becomes nontrivial for higher-rank bundles, where one seeks vector bundles whose curvature is proportional to the identity endomorphism.

\medskip

\noindent
The geometric meaning of projective flatness can be expressed in terms of connections.  
If $\nabla$ is a connection on $\mathcal{V}$ with curvature $F_\nabla$, then $\mathcal{V}$ is projectively flat if and only if
\[
F_\nabla = \omega \cdot \mathrm{Id}_{\mathcal{V}},
\]
for some scalar-valued $(1,1)$-form $\omega$ on $X$.  
This condition implies that the induced connection on the projectivization $\mathbb{P}(\mathcal{V})$ is flat.  
Equivalently, the holonomy representation of $\nabla$ factors through a representation
\[
\rho: \pi_1(X) \longrightarrow PGL_n(\mathbb{C}),
\]
rather than through $GL_n(\mathbb{C})$ itself.

\medskip

\noindent
Projectively flat bundles arise naturally in complex geometry and representation theory.  
A fundamental result of Narasimhan and Seshadri \cite{NarasimhanSeshadri}, reinterpreted by Donaldson \cite{DonaldsonNS}, establishes that on a compact Riemann surface, every stable holomorphic vector bundle is automatically projectively flat.   
This result provides a deep connection between differential geometry, gauge theory, and the moduli theory of stable vector bundles.

\medskip

\noindent
In what follows, we will make use of this structure to globalize the local Wronskian constructions discussed earlier.  
Given a projectively flat bundle $\mathcal{V}$, we will consider local sections $A_i$ on the open sets $U_i$ related by the projectively flat transition rule $\Phi_{ij} = \lambda_{ij} T_{ij}$, and we will study how their local Wronskians assemble into globally defined objects on $X$.

\section{How to Divide Sections of Projectively Flat Vector Bundles}

Let $X$ be a Riemann surface, and let $\mathcal{V}$ be a projectively flat (pf) vector bundle of rank $n$ over $X$.  
We now introduce a notion of \emph{division} between meromorphic sections of $\mathcal{V}$, generalizing the Wronskian quotient defined earlier in the local analytic setting.

\begin{definition}
Let $A$ and $B$ be two meromorphic sections of $\mathcal{V}$, with $B$ generic.  
We define their \emph{Wronskian quotient} as a meromorphic section of the endomorphism bundle $\mathrm{End}(\mathcal{V})$, locally given by
\[
\frac{A}{B} = W(B)^{-1} W(A): U \longrightarrow \mathrm{End}(\mathcal{V}),
\]
on a local chart $U$ where $\mathcal{V}$ admits a projectively flat trivialization.
\end{definition}

This definition generalizes the analytic notion introduced earlier, but now is interpreted in the geometric setting of vector bundles. The quotient $\frac{A}{B}$ measures, in a coordinate-free sense, how the derivatives of $A$ are linearly expressed in terms of those of $B$.

\begin{lemma}
The above definition is independent of the choice of local chart and projectively flat trivialization.  
Furthermore, for any three generic meromorphic sections $A, B,$ and $C$ of $\mathcal{V}$, the following algebraic properties hold:
\[
\frac{A}{B} \cdot \frac{C}{A} \cdot \frac{B}{C} = 1,
\]
and for any nonzero meromorphic function $f$ on $X$,
\[
\frac{fA}{fB} = \frac{A}{B}.
\]
\end{lemma}

\begin{proof}[Sketch of Proof]
The proof is local and follows directly from the corresponding results for meromorphic maps discussed earlier.  
Transition functions in a projectively flat trivialization act by constant matrices up to scalar factors, which cancel in the quotient, ensuring independence of the local frame.
\end{proof}

\medskip
\section{The Wronskian Line Bundle}

In this section, we associate to each projectively flat vector bundle $\mathcal{V}$ on a Riemann surface $X$ a naturally defined line bundle, denoted $w(\mathcal{V})$, which we call the \emph{Wronskian line bundle}.  
This line bundle encapsulates the global behavior of Wronskians of local sections of $\mathcal{V}$ and provides an intrinsic geometric invariant.

If $\mathcal{V}$ is a line bundle, then its Wronskian line bundle is simply itself, i.e.\ $w(\mathcal{V}) = \mathcal{V}$.  
For higher rank bundles, we proceed as follows.

Let $A$ be a generic meromorphic section of $\mathcal{V}$. Although $w(A) = \det W(A)$ depends on both the choice of coordinate on $X$ and the local projectively flat frame on $\mathcal{V}$, the \emph{divisor} of $w(A)$ is globally well defined.  
We denote this divisor by $\mathrm{div}(A)$.

\begin{lemma}
The class of $\mathrm{div}(A)$ in the Picard group $\mathrm{Pic}(X)$ is independent of the choice of generic section $A$.
\end{lemma}

\begin{proof}
For any two generic sections $A$ and $B$, we have
\[
\mathrm{div}(A) - \mathrm{div}(B) = \mathrm{div}\left( \frac{w(A)}{w(B)} \right).
\]
Since $\frac{w(A)}{w(B)} = \det\left( \frac{A}{B} \right)$ is a globally defined meromorphic function, its divisor is principal, and the two divisors define the same class in $\mathrm{Pic}(X)$.
\end{proof}

\begin{definition}
For a projectively flat vector bundle $\mathcal{V}$ on a Riemann surface $X$, the \emph{Wronskian line bundle} $w(\mathcal{V})$ is defined as the line bundle corresponding to the divisor class $\mathrm{div}(A)$ for any generic meromorphic section $A$ of $\mathcal{V}$.
\end{definition}

\subsection{The First Chern Class of the Wronskian Line Bundle}

Recall that for a holomorphic line bundle $L$ with transition cocycle $\lambda_{ij}$ on an open cover $\mathcal{U} = \{U_i\}$, its first Chern class $c_1(L)$ can be represented in \v{C}ech cohomology $\check{H}^1(\mathcal{U}, \Omega^1)$ by the cocycle $\{d\log \lambda_{ij}\}$.

Let $\mathcal{V}$ be a projectively flat bundle on $X$, and let $w(\mathcal{V})$ be its Wronskian line bundle, constructed from a generic section $A$.  
Choose a projectively flat trivialization $\mathcal{U} = \{U_i\}$ such that $w(A)$ has no zeros or poles on overlaps $U_i \cap U_j$.

Then the transition functions of $w(\mathcal{V})$ are given by
\[
\lambda_{ij} = \frac{w(A_i)}{w(A_j)},
\]
where $A_i = A|_{U_i}$.  
On each $U_i$, the section $A_i$ satisfies a local linear differential equation
\[
A_i^{(n)} = p_1^i A_i^{(n-1)} + \cdots + p_{n-1}^i A_i,
\]
for some meromorphic functions $p_k^i$.

By Abel’s identity, we have on overlaps:
\[
d\log \lambda_{ij} = p_1^i - p_1^j.
\]
The collection $\{p_1^i\}$ therefore defines a \v{C}ech 1-cocycle with values in $\Omega^1$, providing a representative for the first Chern class of the Wronskian line bundle:
\[
c_1(w(\mathcal{V})) = [\{p_1^i - p_1^j\}] \in \check{H}^1(\mathcal{U}, \Omega^1).
\]
This provides a geometric interpretation of Abel's identity as the curvature (first Chern class) of the Wronskian line bundle. For a simple algebraic interpretation of Abel's identity, via a duality, see \cite{AjoodanianAbel}.

\medskip

\section{The Wronskian Vector Bundles}

A natural question arises: can the Wronskian line bundle $w(\mathcal{V})$ of a projectively flat bundle $\mathcal{V}$ be realized as the determinant of a higher-rank bundle, a Wronskian vector bundle, naturally associated to $\mathcal{V}$? 
The answer is affirmative, though not unique. In fact, we expect that different choices involved in constructing various 
Wronskian vector bundles fit together to form a \emph{fibration} over the moduli space of pf bundles. But that is a story for another day. We show the construction of a Wronskian vector bundle below.

\medskip

\noindent
Let $\mathcal{V}$ be a pf bundle of rank $n$, and $A$ a generic section of $\mathcal{V}$.  
We construct a new vector bundle $W_A(\mathcal{V})$ of rank $n$ whose determinant line bundle is $w(\mathcal{V})$.  
Different choices of $A$ lead to bundles differing by a twist associated with a meromorphic $n$-th power.

\medskip

Let $\mathcal{U} = \{U_i\}$ be a cover of $X$ trivializing $\mathcal{V}$ in a projectively flat way and such that $w(A)$ has no zeros or poles on $U_i \cap U_j$.  
Let $A_i = A|_{U_i}$, and define the cocycle
\[
\Phi_{ij}^A: U_i \cap U_j \longrightarrow GL_n(\mathbb{C}), \qquad
\Phi_{ij}^A = W(A_j)^{-1} W(A_i).
\]
This satisfies the cocycle condition, hence defines a holomorphic vector bundle $W_A(\mathcal{V})$.  
While $W_A(\mathcal{V})$ is independent of the chosen cover, it may depend on the choice of section $A$.

\begin{lemma}
Let $A$ and $B$ be two generic sections of a projectively flat bundle $\mathcal{V}$ of rank $n$ such that
\[
\mathrm{div}(A) - \mathrm{div}(B) = \mathrm{div}(f^n)
\]
for some meromorphic function $f$ on $X$.  
Then there exists a natural isomorphism
\[
W_A(\mathcal{V}) \simeq W_B(\mathcal{V}).
\]
\end{lemma}

\begin{proof}
Choose a cover $\mathcal{U} = \{U_i\}$ trivializing $\mathcal{V}$ in a projectively flat way such that both $A$ and $B$ have no zeros or poles on $U_i \cap U_j$.  
The cocycles defining $W_A(\mathcal{V})$ and $W_B(\mathcal{V})$ are given by $\Phi_{ij}^A$ and $\Phi_{ij}^B$, respectively.  
We compute:
\[
W(A_j)^{-1} W(A_i)
= W(A_j)^{-1} W(fB_j) \, W(fB_j)^{-1} W(fB_i) \, W(fB_i)^{-1} W(A_i).
\]
This yields
\[
\Phi_{ij}^A = \left( \frac{fB}{A} \right)_j \Phi_{ij}^B \left( \frac{A}{fB} \right)_i.
\]
Since $\mathrm{div}(A) - \mathrm{div}(B) = \mathrm{div}(f^n)$, the quotient $\frac{A}{fB}$ is holomorphic on each $U_i$, and therefore the two cocycles differ by a coboundary.  
This defines a canonical isomorphism between $W_A(\mathcal{V})$ and $W_B(\mathcal{V})$.
\end{proof}

\medskip

\section{Examples}

We now compute the Wronskian line bundle in several basic examples.

\begin{lemma}
If $\mathcal{V}$ is the trivial bundle of rank $n$ on a Riemann surface $X$, then
\[
w(\mathcal{V}) = \frac{n(n-1)}{2} K,
\]
where $K$ denotes the canonical line bundle of $X$.
\end{lemma}

\begin{proof}
Let $f$ be a nonconstant meromorphic function on $X$.  
Then the section
\[
A = (1, f, f^2, \dots, f^{n-1})
\]
is generic.  
By the classical formula for the Wronskian of monomials, we have
\[
w(A) = 1! \, 2! \, \dots (n-1)! \, (f')^{n(n-1)/2}.
\]
Since $f'$ is a local section of the canonical bundle $K$, the result follows.
\end{proof}

\begin{lemma}
Let $\mathcal{V}$ be a projectively flat bundle of rank $n$ on a Riemann surface $X$, and let $L$ be a line bundle on $X$.  
Then
\[
w(\mathcal{V} \otimes L) = w(\mathcal{V}) \otimes L^n.
\]
\end{lemma}

\begin{proof}
Let $A$ be a generic section of $\mathcal{V}$ and $f$ a generic section of $L$.  
Then $fA$ is a section of $\mathcal{V} \otimes L$, and the Wronskian satisfies
\[
w(fA) = f^n w(A).
\]
This yields the stated formula.
\end{proof}

\begin{corollary}
For a line bundle $L$ on a Riemann surface $X$ and $\mathcal{V} = \bigoplus^n L$, one has
\[
w(\mathcal{V}) = nL + \frac{n(n-1)}{2} K.
\]
\end{corollary}

\medskip

\noindent\textbf{Question.} Is it true in general that for any projectively flat vector bundle $\mathcal{V}$ of rank $n$,
\[
w(\mathcal{V}) \stackrel{?}{= }\det(\mathcal{V}) + \frac{n(n-1)}{2} K \, ?
\]
Either answer yes or no is interesting to us. A Yes would provide a direct and elegant relation between the Wronskian line bundle, the determinant bundle, and the canonical line bundle of the Riemann surface, and No shows that the Wronskian line bundle cannot be expressed in terms of such familiar line bundles.

\end{document}